\documentclass[11pt, a4paper, reqno]{amsart}
\usepackage[a4paper, total={6in, 9in}]{geometry}
\usepackage{amsfonts}   
\usepackage{amsmath}
\usepackage[utf8]{inputenc}
\usepackage[english]{babel}
\usepackage{amssymb}
\usepackage{tikz-cd}
\usepackage{comment}
\usepackage{amsthm}
\usepackage{xcolor}
\usepackage[normalem]{ulem}
\newtheorem{theorem}{Theorem}[section]
\newtheorem{cor}{Corollary}[section]
\newtheorem{proposition}{Proposition}[section]
\newtheorem{example}{Example}[section]
\newtheorem{lemma}{Lemma}[section]

\newtheorem*{claim*}{Claim}
\newtheorem*{proofclaim*}{Proof of Claim}

\newtheorem{defi}{Definition}[section]

\newcommand{\df}{\dfrac}
\theoremstyle{theorem}
\newtheorem{thm}{Theorem}[section]

\usepackage{embedfile}

\newtheorem*{Exercise 11.3.B}{Exercise 11.3.B}

\numberwithin{equation}{section}

\usepackage{fancyhdr}

\usepackage{graphicx}
\graphicspath{ {./images/} }

\usepackage{hyperref}                                    
\usepackage{enumerate}

\usepackage[all,cmtip]{xy}

\usepackage{blindtext}

\begin{document}
\title[McIntosh's Conjecture]
{Proofs of McIntosh's Conjecture on Franel Integrals and Two Generalizations}
\author{Bruce C.~Berndt, Likun Xie, Alexandru Zaharescu}
\address{Department of Mathematics, University of Illinois, 1409 West Green
Street, Urbana, IL 61801, USA} \email{berndt@illinois.edu}
\address{Department of Mathematics, University of Illinois, 1409 West Green
Street, Urbana, IL 61801, USA} \email{likunx2@illinois.edu}
\address{Department of Mathematics, University of Illinois, 1409 West Green
Street, Urbana, IL 61801, USA; Institute of Mathematics of the Romanian
Academy, P.O.~Box 1-764, Bucharest RO-70700, Romania}
\email{zaharesc@illinois.edu}

\begin{abstract} We provide a proof of a conjecture made by Richard McIntosh in 1996 on the values of the Franel integrals, $$\int_0^1((ax))((bx))((cx))((ex))\,dx,$$ where $((x))$ is the first Bernoulli function defined in \eqref{B1} below.  Secondly, we extend our ideas to prove a similar theorem for
$$\int_0^1((a_1x))((a_2x))\cdots ((a_{n}x))\,dx.$$
  Lastly, we prove a further generalization in which $((x))$ is replaced by any particular Bernoulli function with odd index.
\end{abstract}
  \keywords{Franel integrals, Bernoulli functions}
\subjclass{11B68}
	\maketitle

\section{Introduction}
\begin{defi} If $h$ and $k$ are relatively prime integers with $k\geq1$, the classical Dedekind sum $s(h,k)$ is defined by
\begin{equation*}
s(h,k):=\sum_{n=1}^{k-1}\left(\bigg(\df{hn}{k}\bigg)\right)\left(\bigg(\df{n}{k}\bigg)\right),
\end{equation*}
where the sawtooth function, or first Bernoulli function, $((x))$ is defined by
\begin{equation}\label{B1}
((x))=\begin{cases} x-\lfloor x \rfloor -\frac12,\quad &\text{if $x$ is not an integer},\\
0, &\text{if $x$ is an integer},
\end{cases}
\end{equation}
where $\lfloor x\rfloor $ is the greatest integer $\leq x$.
\end{defi}
Besides their appearance in the modular transformation formula for the Dedekind eta function \cite[p.~2]{rg}, they enjoy a rich theory  \cite{rg}.  Furthermore, there exist many generalizations and analogues of the Dedekind sum $s(h,k)$, many of which appear in the modular transformation formulas of other modular forms.

Dedekind sums satisfy a beautiful reciprocity theorem.

\begin{theorem}\label{reciprocity} Let $h$ and $k$ denote coprime, positive integers.  Then,
\begin{equation*}
s(h,k)+s(k,h)=-\df14+\df{1}{12}\left(\df{h}{k}+\df{1}{hk}+\df{k}{h}\right).
\end{equation*}
\end{theorem}
There are several proofs of Theorem \ref{reciprocity}, but one of them  \cite{rad} or \cite[p.~25]{rg}, \cite[pp.~369--372]{rademachercp} uses an integral evaluation due to J.~Franel \cite{franel}, namely,
\begin{equation}\label{franel}
\int_0^1((ax))((bx))dx=\df{(a,b)^2}{12ab},
\end{equation}
where $(a,b)$ denotes the greatest common divisor of the positive integers $a$ and $b$.

Although \eqref{franel} is not difficult to prove, it is natural to ask if
\begin{equation}\label{franel4}
I(a,b,c,e):=\int_0^1((ax))((bx))((cx))((ex))\,dx,
\end{equation}
can be evaluated, where $a,b,c$, and $e$ denote positive integers. An evaluation of \eqref{franel4} gives an evaluation of the reciprocal sum [P.193, \cite{mcintosh}]
\begin{equation}
	L(a,b,c,e) = \sum_{\stackrel{as+bt+cu+ev=0}{s,t,u,v\in \mathbb{Z}\setminus\{0\}}}\frac{1}{stuv}= 16\pi^4 I(a,b,c,e)
\end{equation}
 R.~J.~McIntosh \cite{mcintosh} found further representations and other properties for  $I(a,b,c,e)$, and `evaluated' special cases in terms of certain generalized Dedekind sums.  Although he was unable to evaluate \eqref{franel4}, he made a fascinating conjecture \cite[p.~194]{mcintosh}. ``Numerical calculations suggest that the function $f(a,b,c,e)$ defined by
\begin{equation}\label{remarkable}
 f(a,b,c,e) :=\frac{240 a^3 b^3c^3e^3 (a,b,c)(a,b,e)(a,c,e)(b,c,e) }{(a,b)^2(a,c)^2 (a,e)^2 (b,c)^2 (b,e)^2(c,e)^2(a,b,c,e)^4} I(a,b,c,e)
 \end{equation}
is integer-valued, but a proof is out of reach."  Our first goal is to prove McIntosh's remarkable conjecture \eqref{remarkable}.

In view of \eqref{franel} and \eqref{franel4}, it is next natural to ask if one can establish arithmetical properties and an evaluation for
\begin{equation}\label{franel2n}
I_n:=I(a_1,a_2,\dots,a_{n}):=\int_0^1((a_1 x))((a_2 x))\cdots ((a_{n}x))dx,
\end{equation}
where $a_1, a_2, \dots, a_{n}$ are positive integers, and $n=2k$ is any even, positive integer. An evaluation of $I_n$ gives an evaluation of the reciprocal sum by \eqref{higher_eq}:
\begin{equation}
	L(a_1,a_2,\dots, a_n)=\sum_{\stackrel{u_i\in\mathbb{Z}\setminus\{0\}}{\sum _{i=0}^n a_iu_i=0}}\frac{1}{u_1u_2\cdots u_{n}}=(-1)^k(2\pi) ^n I_n.
\end{equation} In Theorem \ref{mcintoshgen} below, we offer and prove such a generalization of McIntosh's conjecture. Note that if $n$ is an odd positive integer in \eqref{franel2n}, by the integrand's asymmetry about $x=\frac12$, $I_n=0$.

The Bernoulli polynomials $B_n(x)$, $n\geq0$, are defined by
\begin{equation}\label{bernoullipoly}
\df{te^{xt}}{e^t-1}=\sum_{n=0}^{\infty}B_n(x)\df{t^n}{n!}, \quad |t|<2\pi.
\end{equation}
In particular, $B_1(x)=x-\frac12$.
The Bernoulli functions $\tilde{B}_n(x)$ are consequently defined by
\begin{equation}\label{bernoullifunction}
 \tilde{B}_n(x) = B_n(x-\lfloor x \rfloor),\quad n\geq0.
 \end{equation}
  The function $((x))$ in \eqref{B1} is thus the first Bernoulli function $B_1(x-\lfloor x \rfloor) , x \notin \mathbb{Z}$.
 The $n^{th}$ Bernoulli number $B_n$ is defined by
\begin{equation}\label{bernoullinumber}
B_n=B_n(0), \quad n\geq 0.
\end{equation}

The next natural question to ask is: Do the theorems arising from McIntosh's conjecture for \eqref{franel4} and the more general integral \eqref{franel2n} have generalizations in which $((x))$ is replaced by $\tilde{B}_n(x)$, $n\geq0$? Indeed, for Bernoulli functions of odd index, in Theorem \ref{theorem_higher_bernoulli} below, we prove such a theorem for
\begin{equation}
	I_{2k+1}(a_1,a_2,\dots, a_{2n}) := \int _0^1 \tilde{B}_{2k+1}(a_1x) \tilde{B}_{2k+1}(a_2x) \cdots \tilde{B}_{2k+1}(a_{2n}x) \, dx.
\end{equation}
which is closely related to the reciprocal sum of $(2k+1)$-th powers by \eqref{higher_eq}:
\begin{equation}
	L_{2k+1}(u_1u_2\cdots u_{2n}) =\sum_{\substack{u_i\in \mathbb{Z}\setminus\{0\}\\ \sum_{i=0}^{2n}a_iu_i=0}}\frac{1}{(u_1u_2\cdots u_{2n})^{2k+1}}=(-1) ^n\left[\frac{(2\pi)^{2k+1}}{(2k+1)!}\right]^{2n}
	I_{2k+1}
\end{equation}
and the reciprocal sum defined by linear forms; see Corollary \ref{reciprocal_sum}.

\section{Proof of McIntosh's Conjecture}
\begin{theorem}\label{mcintosh}
	Let
$$I(a,b,c,e):=\int _0 ^1 ((ax))((bx))
	((cx)) ((ex))\,dx.$$
Then,
$$f(a,b,c,e) :=\frac{240 a^3 b^3c^3e^3 (a,b,c)(a,b,e)(a,c,e)(b,c,e) }{(a,b)^2(a,c)^2 (a,e)^2 (b,c)^2 (b,e)^2(c,e)^2(a,b,c,e)^4} I(a,b,c,e)$$ is an integer for any  positive integers $a,b,c,e$.
\end{theorem}

\begin{proof}
	By Theorem 1 (iii) in \cite{mcintosh}, it suffices to prove Theorem \ref{mcintosh} in the case $a,b, c,e$ are triplet-wise relatively prime, that is, $(a,b,c)= (a,b,e)= (a,c,e)= (b,c,e) =1$.
	Then the claim in the theorem reduces to showing that
$$f(a,b,c,e) =\frac{240 a^3 b^3c^3e^3  }{(a,b)^2(a,c)^2 (a,e)^2 (b,c)^2 (b,e)^2(c,e)^2} I(a,b,c,e)$$
is an integer. We first prove that $$\operatorname{lcm}(a,b,c,e) f(a,b,c,e) = \frac{240 a^4 b^4c^4e^4  }{(a,b)^3(a,c)^3 (a,e)^3 (b,c)^3 (b,e)^3(c,e)^3} I(a,b,c,e)$$ is an integer. (Throughout the sequel, $\operatorname{lcm}$ denotes the least common multiple.) Then we will show that this integer is divisible by each of $a,b,c$ and $e$ separately, and hence divisible by $\operatorname{lcm}(a,b,c,e) $.

	It will be convenient to use the notations  $$L:=\mathrm{lcm}(a,b,c,e)$$ and $$M:= \frac{240 a^3 b^3c^3e^3  }{(a,b)^2(a,c)^2 (a,e)^2 (b,c)^2 (b,e)^2(c,e)^2}. $$
	
	To proceed, observe that the integrand has discontinuities at the  points $1/a, 2/a, \dots , 1/b, 2/b,\dots ,$
	$ 1/c,2/c, \dots,1/e,2/e, \dots$. Denote by $X$ this finite set of points. Next, arrange the elements of $X$ in increasing order and denote them by $x_1, x_2,\dots ,x_n$.   These points break the interval $\left[0,1\right]$ into subintervals.   Accordingly, the integral $I(a,b,c,e)$ can be written as a sum of integrals $I_j$ integrated over the intervals $(x_j,x_{j+1})$. Between any two consecutive points $x_j$ and $x_{j+1}$, the integrand of $I_j$ is given by  a polynomial of degree $4$, denoted by $P_j(x)$, which has the following form:
	\begin{align*}
		P_j(x)=&\, (ax-n_{j1} -1/2 )(bx-n_{j2} -1/2 )(cx-n_{j3} -1/2 )(ex-n_{j4} -1/2 )\\
		=&\,abce x^4  -\left[(abc)(n_{j4}+1/2)+(abe)(n_{j3}+1/2)+(ace)(n_{j2}+1/2)+(bce)(n_{j1}+1/2)\right]x^3\\&+ \left[(ab)(n_{j3} +1/2) (n_{j4}+1/2)+\dots \right]x^2 -\left[a(n_{j2}+1/2 ) (n_{j3} +1/2)(n_{j4} +1/2)+\dots \right] x\\
		&+ (n_{j1 } +1/2) (n_{j2} +1/2)(n_{j3} +1/2)(n_{j4} +1/2),
	\end{align*}
	where $n_{j1} ,  n_{j2}, n_{j3}, n_{j4}$ are  positive integers. Next, integrate each $P_j(x) $  over the corresponding subinterval. Denote by $F_j(x) $ an antiderivative of $P_j(x)$. Then,
	\begin{align}\label{fj}
		F_j(x)
		=&\frac{1}{5}abce x^5  -\frac{1}{4}\left[(abc)(n_{j4}+1/2)+(abe)(n_{j3}+1/2)+(ace)(n_{j2}+1/2)+(bce)(n_{j1}+1/2)\right]x^4\nonumber\\&- \frac{1}{3}\left[(ab)(n_{j3} +1/2) (n_{j4}+1/2)+\dots \right]x^3 +\frac{1}{2}\left[a(n_{j2}+1/2 ) (n_{j3} +1/2)(n_{j4} +1/2)+\dots \right] x^2\\
		\nonumber&+ (n_{j1 } +1/2) (n_{j2} +1/2)(n_{j3} +1/2)(n_{j4} +1/2)x.
	\end{align}

 Next, we compute $I_j$ by evaluating $F_j(x)$ at the two endpoints $x_j$ and $x_{j+1}$. Recall that $I(a,b,c,e)$ is the sum of all the integrals $I_j$'s.

 Moreover, the coefficient of $x^k$ in \ref{fj} multiplied by $k\cdot2^{5-k}$, $1\leq k\leq 5$,  is a polynomial in $a,b,c,e$ with integer coefficients. Note that
 $240$ is divisible by each of $k\cdot2^{5-k}$, $1\leq k\leq 5$. So after multiplying by $240$, we have $240 F_j(s)$ is a sum of
 integer multiples of  $a^ib^jc^ke^lx^m$ for $i+j+k+l=m-1$, $0\leq i,j,k,l\leq 1$.

Theorem \ref{mcintosh} will then follow from the two lemmas below.

	\begin{lemma}
		$LM\, I(a,b,c,e)=L\,f(a,b,c,e)$ is an integer.
	\end{lemma}

	\begin{proof}
		We want to show that $LM\, F_j(x)$ evaluated at $x_j$ and $x_{j+1}$ is an integer.
		Since each of the endpoints is an integer multiple of either $\frac{1}{a}$, $\frac{1}{b}$, $\frac{1}{c}$ or $\frac{1}{e}$, by symmetry it suffices to show that $LM\,F_j(x)$ evaluated at $\frac{1}{a}$ is an integer. For each degree, $1\leq m\leq 5$, in the representation of $F_j(x)$, view  the coefficient of $x^m$ as a polynomial in $a,b,c,e$. Consider a monomial of the form $a^ib^jc^ke^l$ with degree $m-1$, i.e., $i+j+k+l=m-1$ in the coefficient of $x^m$, where $0\leq i,j,k,l\leq 1$.  Observe that $\frac{a^ib^jc^ke^l}{a^{m}}$ is an integer multiple of
$$\frac{1}{a\left[\frac{a}{(a,b)}\right]^j\left[\frac{a}{(a,c)}\right]^k\left[\frac{a}{(a,e)}\right]^l}$$
 whose denominator divides $LM$. This is because  $$LM=240\,\mathrm{lcm}(a,b,c,e) \frac{ a}{(a,b)}\frac{b}{(a,b)}\cdots \frac{c}{(c,e)}\frac{e}{(c,e)}$$ which is divisible by $$a\frac{a}{(a,b)}\frac{ a}{(a,c)}\frac{ a}{(a,e)}.$$
	
		Since $LMF_j(\frac{1}{a})$ is a sum of terms of the form  $LM\frac{a^ib^jc^ke^l}{a^m}$ with $i+j+k+l=m-1$, $0\leq i,j,k\leq 1$ which we have shown to be integers,
		 we conclude that   $LM\, I(a,b,c,e)=L\,f(a,b,c,e)$ is an integer.
	\end{proof}
	
	\begin{lemma}\label{lemma2.2}
		$LM\, I(a,b,c,e)=L\,f(a,b,c,e)$ is divisible by each of $a,b,c$ and $e$.
	\end{lemma}
	\begin{proof}
		Without loss of generality, we show that $L \,I(a,b,c,e)$ is divisible by $a$. It suffices to show that $LM\,F_j(x) $ evaluated at $\frac{1}{a}, \frac{1}{b}, \frac{1}{c}$ and $\frac{1}{e}$ is divisible by $a$. We first show that $LM\,F_j(x)$ evaluated at $\frac{1}{b}, \frac{1}{c}$ and $\frac{1}{e}$ is divisible by $a$. Without loss of generality, it suffices to show that $LM\,F_j(x)$ evaluated at $\frac{1}{b}$ is divisible by $a$. As before, it suffices to consider a monomial of the form $a^ib^jc^ke^l$ with degree $m-1$ in the coefficient of $x^m$, $0\leq i,j,k,l\leq 1$.  Evaluated at $\frac{1}{b }$, we see that $\frac{a^ib^jc^ke^l}{b^m}$ is an integer multiple of $$\frac{1}{b\left[\frac{b}{(b,a)}\right]^i\left[\frac{b}{(b,c)}\right]^j\left[\frac{b}{(b,e)}\right]^l}.$$
 Multiplying this by $LM$, we find that $LM\,\frac{a^ib^jc^ke^l}{b^m}$ is an integer multiple of $$\frac{a}{(a,b)}\frac{a}{(a,c)}.$$  As $((a,b),(a,c))=1$, $\frac{a}{(a,b)}\frac{a}{(a,c)}$ is divisible by $a$. Hence, $LM\,\frac{a^ib^jc^ke^l}{b^m}$ is divisible by $a$.
		
		Next, we show that $LMF_j(x)$ evaluated at $\frac{1}{a}$ is divisible by $a$. Consider a monomial  $a^ib^jc^ke^l$   in the coefficient of $x^m$  in $F_j(x)$ with degree $1\leq m\leq 4$ and $i+j+k+l=m-1$,  $0\leq i,j,k,l\leq 1$. If $i=1$,  then
 $$		\frac{a^ib^jc^ke^l}{a^m} =\frac{b^ic^ke^l}{a^{m-1}} $$
  is an integer multiple of
  $$\frac{1}{a\left[\frac{a}{(a,b)}\right]^j\left[\frac{a}{(a,c)}\right]^k\left[\frac{a}{(a,e)}\right]^l}.$$
  		Since $j+k+l=m-2\leq 2$, one of $j,k,l$ is $0$. Without loss of generality, assume $l=0$. Then multiplying by $LM$, we see that $LM\,\frac{a^ib^jc^ke^l}{a^m}$ is an integer multiple of  $$\frac{a}{(a,e)}\frac{e}{(e,b)},$$
		which is divisible by $a$ since $(a,e)$ and $(e,b)$ are coprime.
		
		If
		$m\leq 3$ and $i=0$, then, since $j+k+l=m-1\leq 2$, one of $j,k,l$ is $0$. We may assume $l=0$. Then $
\frac{b^jc^k}{a^m} $ is an integer multiple of $$\frac{1}{a\left[ \frac{a}{(a,b)}\right]^j\left[\frac{a}{(a,c)}\right]^k}.$$ Then multiplying by $LM$, we conclude that $LM\,\frac{a^ib^jc^ke^l}{a^m}$ is also an integer multiple of  $$\frac{a}{(a,e)}\frac{e}{(e,b)},$$
which is divisible by $a$ since $(a,e)$ and $(e,b)$ are coprime.
		
		Therefore, viewing the coefficients of $x^k$ in $F_j(x)$ as polynomials in $a,b,c$ and $e$,  since adding or subtracting terms divisible by $a$ does not change divisibility of $a$, the following terms can be  removed from further consideration in $F_j(x)$:  the terms involving the monomials $ab^jc^ke^l$ with $j+k+l=2$ in the coefficient of $x^4$, and of the form $a^ib^jc^ke^l$ with $i+j+k+l=m-1$ in the coefficient of $x^m$ with $m\leq 3$. This implies that we can ignore $n_{j2} ,  n_{j3} , n_{j4}$ in $F_j$, since adding or subtracting terms divisible by $a$ does not change the divisibility of $a$. Furthermore, since $2^4\mid 240$, we can remove $1/2$ in $F_j$ as well for the same reason.  Now it remains  to show that
$$LM\int_0^1 ((ax))bxcxex\, dx$$
 is divisible by $a$. Evaluating the integral,
		we have
		\begin{align}
			\int_0^1 ((ax))bxcxex\, dx&=
			\sum_{k=0} ^{a-1} \int _{\frac{k}{a}} ^{\frac{k+1}{a}} (ax-k-1/2) bxcxex\, dx\nonumber \\\label{integral1}
			&=\sum_{k=0} ^{a-1} \int _{\frac{k}{a}} ^{\frac{k+1}{a}} -bcekx^3\, dx +\int_0^1bce(ax-1/2) x^3\,dx.
		\end{align}
		It suffices to consider the first term on the right-hand side of \eqref{integral1}, since the second term multiplied by $LM$ is divisible by $a$.  We have
		\begin{align}
			\sum_{k=0} ^{a-1} \int _{\frac{k}{a}} ^{\frac{k+1}{a}} -bcekx^3\, dx  &=  \frac{ bce}{4}\left[\frac{1^4+2^4+\dots +(a-1)^4}{a^4}-(a-1)\right]\nonumber\\&=\frac{bce}{4a^4}\cdot \frac{a(a-1) \left[6(a-1)^3+9(a-1)^2+a-2\right]}{30} -\frac{bce(a-1)}{4}.\label{eqsum1}
		\end{align}
	 We can ignore the second term  $-\frac{bce(a-1)}{4}$ in \eqref{eqsum1} since $4\mid 240$. The first term is an integer multiple of $\frac{bce}{120a^3}$ and hence a multiple of
$$\frac{1}{120\frac{a}{(a,b)}\frac{a}{(a,c)}\frac{a}{(a,e)}}, $$ whose denominator divides $M$. Multiplying by $LM$, since $a_1|L$, we conclude that $$\frac{LM}{120\frac{a}{(a,b)}\frac{a}{(a,c)}\frac{a}{(a,e)}} $$ is   divisible by ${a}$.  This concludes the proof of Lemma \ref{lemma2.2} and therefore also that of Theorem \ref{mcintosh}.
	\end{proof}
\end{proof}

	Note that the constant in McIntosh's conjecture is sharp, as seen  from the following example.
	\begin{example}\label{example_mcintosh} We have
		 $$I(a,1,1,1)  = \frac{5a-2}{240a^3},$$
 and the constant in McIntosh's conjecture is $240a^3$. When $a=3k$, where $k$ an odd integer, the denominator of $I(a,1,1,1)$ is exactly $240a^3$.
	\end{example}

\section{Generalization of McIntosh's Conjecture}
More generally, for $n=2k$, define $$
I_n=I(a_1,\dots,a_n)=\int_0^1((a_1x))((a_2x) ) \cdots ((a_nx)) dx.  $$

\begin{defi}
	Let $f\in \mathbb{Q}[x]$ be a rational polynomial, \textbf{the denominator of $f$} is the smallest positive integer $N\in \mathbb{Z}_{>0}$ such that $Nf\in \mathbb{Z}[x]$.
\end{defi}

Recall that the Bernoulli polynomials are defined in \eqref{bernoullipoly}. See  \cite[Theorem 4]{power-sum} for an expression of the denominator of the Bernoulli polynomials.

\begin{thm}\label{mcintoshgen}
	Let $n=2k$, and let $B$ be the denominator of the  polynomial $B_{n+1}(x)$ [Theorem 4, \cite{power-sum}]. Then
$$	f_n:=  \frac{\mathrm{lcm}(2^{2k}\prod_{j=1}^k(2j+1), nB)
		 X_1^{2k-1}X_3^{2k-3}\cdots  X_{2k-1}}{X_2^2X_4^4\cdots  X_{2k}^{2k}}I_{2k}$$
	is an integer, where $$X_m= \prod_{1\leq i_1<i_2<\dots < i_m\leq n=2k}(a_{i_1},a_{i_2},\dots ,  a_{i_m})$$ is the product of all greatest common divisors of $a_{i_1},a_{i_2},\dots ,  a_{i_m}$, with $1\leq i_1<i_2<\dots < i_m\leq n=2k$.
\end{thm}

\begin{example}
	When $k=1$, the theorem states that $$12
	\frac{ab}{(a,b)^2}I_2(a,b)$$ is an integer, in agreement with \eqref{franel}.
\end{example}
\begin{example}
		When $k=2$,  the theorem reduces to McIntosh's conjecture, Theorem \ref{mcintosh}.

\end{example}

\begin{example}
When $k=3$, the theorem states that
  $$\frac{4032X_1^5X_3^3X_5}{X_2^2X_4^4X_6^6} I_6$$ is an integer,
where $$X_m= \prod_{1\leq i_1<i_2<\dots < i_m\leq 6}(a_{i_1},a_{i_2},\dots ,  a_{i_m})$$
 is the product of all greatest common divisors of $a_{i_1},a_{i_2},\dots ,  a_{i_m}$, with $1\leq i_1<i_2<\dots < i_m\leq 6$.

Explicitly,
\begin{align*}
	&X_1=a_1a_2\cdots a_6,\\
	&X_2=\prod_{1\leq i<j\leq 6} (a_i,a_j),\\
	&X_3=\prod_{1\leq i<j<k\leq 6} (a_i,a_j,a_k),\\&\dots\\
	&X_6=(a_1,a_2,\dots,a_6).
\end{align*}
\end{example}

\begin{proof}
	First, we reduce to the case that $(a_1,a_2\dots ,a_{n})=1$ by Proposition \ref{prop_gcd}. Let
$$L:=\mathrm{lcm}(a_1,\dots, a_n) $$ and
\begin{equation}\label{M}
M:=\frac{\mathrm{lcm}(2^{2k}\prod_{j=1}^k(2j+1), nB) X_1^{2k-1}X_3^{2k-3}\cdots  X_{2k-1}}{X_2^2X_4^4\cdots  X_{2k}^{2k}}.
\end{equation}
 Next, we show that $L\,f_n$ is an integer. Lastly, we show that $L\,f_n$ is divisible by each of $a_1,\dots ,a_n$ separately, and thereby conclude that $f_n$ is an integer.

	The integrand of $I_n$ has discontinuities at the points $  \frac{ 1}{a_k}, \frac{2}{a_k},\dots, \frac{a_k-1}{a_k}$, for $1\leq k\leq n $. Denote by $S$ this finite set of points. Order the elements in $S$ by $x_1,x_2,\dots, x_N$.  These points break the interval $\left[0,1\right]$ into sub-intervals, and so the integral $I_n$ is represented as a sum of integrals over these sub-intervals. Let $S=\{1,2,\dots, n\}$ be the index set. Between two consecutive points $x_j$ and $x_{j+1}$, the integrand is given by the polynomial $P_j(x)$ as follows:
	\begin{align*}
P_j(x)=&\left(\prod_{i=1} ^n a_i\right) x^{n} -\left[\sum_{k=1}^n\left(\prod_{\substack{i=1\\i\neq k}}^na_i\right)\left(n_{jk}+\frac{1}{2}\right)\right]x^{n-1} \notag\\ &+\left[\sum_{\substack{k\neq l\\\{k,l\}\subset S}}\prod_{\substack{i=1\\i\neq k,l}}^n a_i\left(n_{jk}+\frac{1}{2}\right)\left(n_{jl}+\frac{1}{2}\right)\right] x^{n-2} -\dots +\prod_{k=1}^n\left(n_{jk}+\frac{1}{2}\right),
\end{align*}
  where the positive integers $n_{jk}$ depend on the endpoints $x_j$. Let $F_j(x)$ denote an antiderivative of $P_j(x)$. Then $I_n$ is the sum of the antiderivatives $F_j(x)$ taken at the end points $x_{j+1}$ and $x_j$, $1\leq j\leq N$. More precisely,
	\begin{align*}
		F_j(x)=&\frac{1}{n+1}\prod_{i=1} ^n a_i x^{n+1} -\frac{1}{n}\left[\sum_{k=1}^n\prod_{\substack{i=1\\i\neq k}}^na_i\left(n_{jk}+\frac{1}{2}\right)\right]x^{n}\notag\\&+ \frac{1}{n-1}\left[\sum_{\substack{k\neq l\\\{k,l\}\subset S}}\prod_{\substack{i=1\\i\neq k,l}}^n a_i\left(n_{jk}+\frac{1}{2}\right)\left(n_{jl}+\frac{1}{2}\right)\right] x^{n-1} -\dots +\prod_{k=1}^n\left(n_{jk}+\frac{1}{2}\right)x.
	\end{align*}

	\begin{lemma}\label{lemma1}
		$L\,f_n=LM\,I_n$ is an integer.
	\end{lemma}

	\begin{proof}
		Without loss of generality, it is sufficient to show that $LM\,F_j(x) $ evaluated at $\frac{1}{a_1}$ is an integer by symmetry (since every endpoint is an integer multiple of $\frac{1}{a_i}$ for some $i$).
		In $F_j(x)$, consider the coefficient of $x^m$ as a polynomial in $a_1,\dots, a_n$. The coefficients of this polynomial in $a_1,\dots, a_n$ have denominators $m2^{n+1-m}$. Since $m\leq 2^{m-1}$, we see that
$$m2^{n+1-m}\mid 2^n\prod_{j=1}^k(2j+1).$$
		Therefore, we only need to consider the monomials in $a_1,\dots,a_n$. Thus, it suffices to show that
\begin{equation}
LM\prod_{l=1}^{r-1}a_{i_l}/{a_1^r}
\end{equation}
is an integer for $1\leq r\leq n+1$, where the $i_l$ are distinct, $1\leq l \leq r-1$. The product above is an integer multiple of  $$\frac{LM}{a_1\prod_{l=1}^{r-1}\frac{a_1}{(a_1, a_{i_l})}}.$$
 Note that
 \begin{equation}\label{factor_of_M}
 	\frac{X_1^{2k-1}}{X_2^2}\bigg| M
 \end{equation}
  and
  \begin{equation}\label{factor1_of_M}
  	\frac{X_1^{2k-1}}{X_2^2}=\prod_{i\neq j}\frac{a_i}{(a_i,a_j)}\frac{a_j}{(a_i,a_j)}.
  \end{equation}
  Thus,
  $$\prod_{l=1}^{r-1}\frac{a_1}{(a_1, a_{i_l})}\left|\frac{X_1^{2k-1}}{X_2^2}.\right.$$
  Moreover, $a_1\mid L$, and so we have finished the proof.
	\end{proof}

	Next, we want to show that $L\,f_n$ is divisible by $a_i$, $1\leq i \leq n$. Again, without loss of generality, it suffices to show that $L\,f_n$ is divisible by $a_1$. We want to show that each $LM\,F_j(x)$ evaluated at the endpoints $x_j$ and $x_{j+1}$ is divisible by $a_1$. It suffices to show that $LM\,F_j(x)$, evaluated at $\frac{1}{a_i}$ is divisible by $a_1$, $1\leq i \leq n$, since each endpoint is an integer multiple of $\frac{1}{a_j}$ for some $1\leq j\leq n$.  We finish  the proof of the theorem by using the following two lemmas.

	\begin{lemma}\label{lemma2}
		Suppose $m\neq 1$. Then  $LM\,F_j(\frac{1}{a_m})$ is divisible by $a_1$.
	\end{lemma}
	
	\begin{proof}
		Again, the denominator of the coefficients of the monomials in $a_1,\dots,a_n$ in the coefficients of $x^i$ all divide the constant factor in $M$. It remains to consider the term
$$LM\,\df{\prod_{l=1}^{r-1}a_{i_l}}{{a_m^r}}$$
for each degree, $1\leq r\leq n+1$. This is an integer multiple of  $$\frac{LM}{a_m\prod_{l=1}^{r-1}\frac{a_m}{(a_m,a_{i_l})}}. $$
		First, note that $a_m|L$. Second, observe that
$$\df{M}{\prod_{l=1}^{r-1}\frac{a_m}{(a_m,a_{i_l})}}$$
 is an integer multiple of
		\begin{align}\label{factor}
			\frac{a_1}{(a_1,a_2)}\frac{a_1}{(a_1,a_3)}  \frac{(a_1,a_2,a_3)}{(a_2,a_3,a_4,a_5)}\cdots  \frac{(a_1,\dots, a_{2k-1})}{(a_2,\dots, a_{2k})},
		\end{align}
		which is divisible by $a_1$. To see this, we can rewrite \eqref{factor} as $$\frac{a_1}{(a_1,a_2)}\frac{a_1}{X},$$ where $$X=\frac{ (a_1,a_3)}{(a_1,a_2,a_3)}\left[\prod _{i=3}^ {k}\frac{(a_2,\dots, a_{2i-1})}{(a_1,a_2,\dots, a_{2i-1})}\right](a_2,a_3,\dots, a_{2k}).$$
		
		Note that
\begin{align*}\left((a_1,a_2),\frac{ (a_1,a_3)}{(a_1,a_2,a_3)}\right)&=1, \\
  \left((a_1,a_2),\frac{(a_2,\dots, a_{2i-1})}{(a_1,a_2,\dots, a_{2i-1})}\right)&=1,\\
   \intertext{and}    \left((a_1,a_2),(a_2,a_3,\dots, a_{2k})\right)&=1.
    \end{align*}
    Thus $\left((a_1,a_2),X\right)=1$, and $(a_1,a_2)$ divides $\frac{a_1}{X}$. Therefore, $a_1$ divides $$\frac{a_1}{(a_1,a_2)}\frac{a_1}{X},$$
     which finishes the proof.
	\end{proof}
	
	\begin{lemma}\label{lemma3}
		$LM\,F_j(\frac{1}{a_1})$ is divisible by $a_1$.
	\end{lemma}
	\begin{proof}
		For each degree $r$ in the polynomial $F_j(x)$, it suffices to show that
$$\df{LM}{\prod_{l=1}^{r-1}a_{i_l}/{a_1^r}}$$
 is divisible  by $a_1$. Let $1\leq r\leq n$, and suppose that $a_1$ appears in $\prod_{l=1}^{r-1}a_{i_l}$;  we may assume $a_{i_1}=a_1$. Then  we can write $$LM\,\frac{\prod_{l=1}^{r-1}a_{i_l}}{a_1^r} = LM\,\frac{\prod_{l=2}^{r-1} a_{i_l}}{a_1^{r-1}},$$
which is an integer multiple of
$$\frac{LM}{a_1\prod_{l=2}^{r-1}\frac{a_1}{(a_1,a_{i_l})}}.$$

Let $I=\{1,2,\dots, n\}$ be the index set. Since $|\{i_2,i_3,\dots, i_{r-1}\}|=r-2\leq n-2$, there exists an integer $m\in I\setminus\{1,i_2,i_3,\dots, i_{r-2}\}$.  By \eqref{factor1_of_M}, we have
		
$$\df{\frac{X_1^{2k-1}}{X_2^2}}{\prod_{l=2} ^{r-1}\frac{a_1}{(a_1,a_{i_l})}}$$
 is an integer multiple of  $\frac{a_1}{(a_1,a_m)}\frac{a_m}{(a_m,a_p)}$ for $p\neq 1,m$. By permuting the index set $I$, we may assume that $m=2$.  Then observe that
 $$\df{M}{\prod_{l=2} ^{r-1}\frac{a_1}{(a_1,a_{i_l})}}$$
  is an integer multiple of
		\begin{align}\label{factor3}
			\frac{a_1}{(a_1,a_2)}\frac{a_2}{(a_2,a_3)}  \frac{(a_1,a_2,a_3)}{(a_2,a_3,a_4,a_5)}\dots  \frac{(a_1,\dots, a_{2k-1})}{(a_2,\dots, a_{2k})}.
		\end{align}
	
	Let
	\begin{equation}\label{factor_Y}
		Y= \prod _{i=2}^ {k}\frac{(a_2,\dots, a_{2i-1})}{(a_1,a_2,\dots, a_{2i-1})}(a_2,a_3,\dots, a_{2k}).
	\end{equation}
 Then we rewrite \eqref{factor3} as $$\frac{a_1}{(a_1,a_2)}\frac{a_2}{Y}.$$ Since
 $$\left((a_1,a_2),\frac{(a_2,\dots, a_{2i-1})}{(a_1,a_2,\dots, a_{2i-1})}\right)=1,\quad \text{and}\quad  \left((a_1,a_2),(a_2,a_3,\dots, a_{2k})\right)=1,$$
  we have $$\left((a_1,a_2),Y\right)=1.$$ Hence,
  $$(a_1,a_2)\mid\frac{a_2}{Y}\quad\text{and}\quad a_1\mid \frac{a_1}{(a_1,a_2)}\frac{a_1}{Y}.$$ So, $$LM\,\frac{\prod_{l=1}^{r-1}a_{i_l}}{a_1^r}$$
  is divisible by $a_1$.

 Suppose that $r\leq n-1$ and $a_1$ does not appear in $\prod_{l=1}^{r-1}a_{i_l}$. Then
 $$LM\,\frac{\prod_{l=1}^{r-1}a_{i_l}}{a_1^r}$$
 is an integer multiple of
 $$\frac{LM}{\prod_{l=1} ^{r-1}\frac{a_1}{(a_1,a_{i_l})}}.$$
 Since $r-1\leq n-2$, there exists an integer $m$ such that $m\in I\setminus \{1,{i_1},{i_2},\dots, {i_{r-1}}\}$. Without loss of generality, we may assume that $m=2$. Then
 $$\df{\frac{X_1^{2k-1}}{X_2^2}}{\prod_{l=1} ^{r-1}\frac{a_1}{(a_1,a_{i_l})}}$$  is an integer multiple of  $\frac{a_1}{(a_1,a_2)}\frac{a_2}{(a_2,a_3)}$, and
 $$ \df{M}{\prod_{l=1} ^{r-1}\frac{a_1}{(a_1,a_{i_l})}}$$ is an integer multiple of  $$\frac{a_1}{(a_1,a_2)}\frac{a_2}{Y},$$ for $Y$ as in \eqref{factor_Y}. In this case, $$LM\,\frac{\prod_{l=1}^{r-1}a_{i_l}}{a_1^r}$$ is also divisiblly by $a_1$.

 Therefore, $$LM\frac{\prod_{l=1}^{r-1}a_{i_l}}{a_1^r}$$
 is divisible by $a_1$ whenever $r\leq n-1$,  or $r=n$ and $1\in \{i_1,\dots, i_{r-1}\}$. 	Hence, in $F_j(x)$ we can ignore the monomials involving $a_1$ in the coefficients of $x^n$ and all other monomials in the coefficients of $x^j$ for $1\leq j\leq n-1$, since adding or substracting a multiple of $a_1$ does not change divisibility of $a_1$. Hence, after removing some terms that are divisible by $a_1$, it suffices to show that $$a_1\mid LM\int_0^1((a_1x))a_2x\cdots  a_nx \,dx.$$

Evaluating the integral, we have
$$ \int_0^1((a_1x))a_2x\cdots  a_nx \,dx =\int _0^1 \left(a_1x-\frac{1}{2}\right)a_2x\cdots a_nx\,dx -a_2\cdots a_n \sum_{j=1}^{a_1-1}\int _{\frac{j}{a_1}}^{\frac{j+1}{a_1}}jx^{n-1}dx,$$
where the first term on the right-hand side multiplied by $LM$ is divisible by $a_1$. For the second term, $$-a
_2\cdots a_n\sum_{j=1}^{a_1-1}\int _{\frac{j}{a_1}}^{\frac{j+1}{a_1}}jx^{n-1}dx=\frac{a_2\cdots a_n}{n}\left[\left(\frac{1}{a_1}\right)^n\sum_{k=1}^{a_1-1} k^n-(a_1-1)\right].$$

It suffices to consider the first term
$$\frac{a_2\cdots a_n}{na_1^n}\sum_{k=1}^{a_1-1} k^n.$$
We can write the sum of $n^{th}$ powers in terms of Bernoulli polynomials and numbers as \cite[p.~589, Section 24.4(iii), no.~24.4.7]{nist}
 \begin{equation}\label{sumofpowers}
 \sum_{k=1} ^{a_1-1} k^n =
\frac{B_{n+1}(a_1)-B_{n+1}}{n+1}=\frac{B_{n+1}(a_1)}{n+1},
\end{equation}
which is a rational polynomial  in $a_1$ whose constant term is $0$.
Therefore,  the denominator of $\frac{\sum_{k=1}^{a_1-1} k^n}{n}$ divides $n(n+1)B$.
Recall that $B$ is defined in the statement of Theorem \ref{mcintoshgen}. By \cite{power-sum},
$$B\bigg|\prod_{\stackrel{p\leq \frac{n+2}{2}}{p \text{ prime }}}p.$$
 Hence,
$$n(n+1)B\Big| \mathrm{lcm}(2^{2k}\prod_{j=1}^k(2j+1), nB).$$

Thus, recalling the definition \eqref{M} of $M$, we find that
 $$\frac{a_2\cdots a_n}{na_1^n}\sum_{k=1}^{a_1-1} k^n \times\mathrm{lcm}(2^{2k}\prod_{j=1}^k(2j+1), nB)$$ is an integer multiple of
 $\frac{a_2\cdots a_n}{a_1^{n-1}}$. Moreover,  $\frac{a_2\cdots a_n}{a_1^{n-1}}$  is an integer multiple of $$\frac{1}{\prod_{i=2}^n \frac{a_1}{(a_1,a_i)}},$$ whose denominator will be cancelled out by $\frac{X_1^{2k+1}}{X_2^2}$ in $M$ after multiplying by $M$. Therefore, $$M\frac{a_2\cdots a_n}{na_1^n}\sum_{k=1}^{a_1-1} k^n$$  is an integer.
    Lastly, since $a_1\mid L$, we conclude that  $$LM\frac{a_2\cdots a_n}{na_1^n}\sum_{k=1}^{a_1-1} k^n$$ is divisible by $a_1$. This finishes the proof of Lemma \ref{lemma3} and Theorem \ref{mcintoshgen} as well. 	\end{proof}
	
\end{proof}
\begin{example}
	Four examples for $I(a,1,1,\dots, 1)$. For $n=4$, see Example \ref{example_mcintosh}.
	
	\textup{(a)} $n=6$. $$I(a,1,\dots,1) = \frac{16-28a^2+21a^4}{4032a^5},$$ $$B=6, $$
	$$4032= 2^6\cdot 3^2\cdot 7 .$$
	Theorem \ref{mcintoshgen} asserts that $I(a,1,\dots,1)$ multiplied by  $\mathrm{lcm}(6\cdot6, 2^6\cdot 7\cdot 5\cdot3)a^5= 2^6\cdot 3^2\cdot 5\cdot 7 a^5$ is an integer.
	
	\textup{(b)} $n=8$. $$I(a,1,\dots,1) =  \frac{-48 + 80 a^2 - 42 a^4 + 15 a^6}{11520 a^7},$$
 $$B=10, $$
	$$11520= 2^8\cdot 3^2\cdot 5 .$$ Theorem \ref{mcintoshgen} implies that $I(a,1,\dots,1)$ multiplied by  $\mathrm{lcm}(10\cdot8, 2^8 \cdot 9\cdot 7\cdot 5\cdot3)a^7= 2^8\cdot 3^3\cdot 5\cdot 7 a^7$ is an integer.
	
		\textup{(c)} $n=10$. $$I(a,1,\dots,1) =  \frac{1280-2112 a^2+1056 a^4-264 a^6+55 a^8}{168960 a^9},$$ $$B=6, $$
	$$168960= 2^{10}\cdot 3\cdot 5\cdot 11 .$$ Theorem \ref{mcintoshgen} asserts that $I(a,1,\dots,1)$ multiplied by  $\mathrm{lcm}(6\cdot10, 2^{10}\cdot 11\cdot 9\cdot 7\cdot 5\cdot3)a^7= 2^{10}\cdot3^3 \cdot 5\cdot 7\cdot11$ is an integer.
	
		\textup{(d)} $n=12$. $$I(a,1,\dots,1) =  \frac{-353792+582400 a^2-288288 a^4+68640 a^6-10010 a^8+1365 a^{10}}{16773120 a^{11}},$$ $$ B=210, $$
	$$16773120= 2^{12}\cdot 3^2\cdot 5\cdot 7\cdot 13 .$$ Theorem \ref{mcintoshgen} implies that $I(a,1,\dots,1)$ multiplied by  $\mathrm{lcm}(210\cdot12, 2^{12}\cdot 13\cdot 11\cdot 9\cdot 7\cdot 5\cdot3)a^7= 2^{12}\cdot3^3 \cdot 5\cdot 7\cdot11\cdot 13$ is an integer.
\end{example}

\section{Generalization to Higher Order Bernoulli Polynomials }

Recall from \eqref{bernoullifunction} that the periodic Bernoulli functions $\tilde{B}_n(x)$ are defined  by $\tilde{B}_n(x)= B_n(\{x\})$, where $\{x\}=x-\lfloor x\rfloor$. The Fourier series for the periodic Bernoulli functions of odd index are given by \cite[p.~592, Equation 24.8.2]{nist}
\begin{align}\label{bernoullifourier}
	\tilde{B}_{2n+1} (x) =(-1) ^{n+1}   \frac{2(2n+1) !}{(2\pi )^{2n+1} } \sum_{m=1} ^\infty \frac{\mathrm{sin} (2m\pi x)}{m^{2n+1}},\quad x\in \mathbb{R}, n\in \mathbb{N}.
\end{align}

Define the integral $I_{k}(a_1,a_2,\dots, a_{2n})$ by
\begin{align}\label{bernoulliintegral}
	I_{k}(a_1,a_2,\dots, a_{2n}) := \int _0^1 \tilde{B}_k(a_1x) \tilde{B}_k(a_2x) \cdots \tilde{B}_k(a_{2n}x) \, dx.
\end{align}
It follows that \begin{align}\label{higher_eq}
	I_{2k+1}& (a_1,a_2,\dots, a_{2n})= \left[\frac{2(2k+1)!}{(2\pi)^{2k+1}}\right]^{2n} \int _0^1 \underset{l\to \infty}{\mathrm{lim}}\sum _{1\leq u_1,u_2,\dots , u_{2n}\leq l}\prod_{i=1}^{2n} \frac{\mathrm{sin}(2u_i\pi a_i x)}{u_i ^{2k+1}}\, dx  \nonumber\\
	&= \frac{(-1) ^n}{2^{2n}}\left[\frac{2(2k+1)!}{(2\pi)^{2k+1}}\right]^{2n} \int_0^1 \underset{l\to \infty}{\mathrm{lim}}\sum _{1\leq |u_1|,|u_2|,\dots , |u_{2n}|\leq l}\frac{\mathrm{cos}\left[2\pi (\sum _{i=1}^{2n}u_ia_i)x\right]}{(u_1u_2\cdots u_{2n})^{2k+1}}\, dx\nonumber\\
	&= (-1) ^n\left[\frac{(2k+1)!}{(2\pi)^{2k+1}}\right]^{2n}\sum_{\substack{u_i\in \mathbb{Z}\setminus\{0\}\\ \sum_{i=0}^{2n}a_iu_i=0}}\frac{1}{(u_1u_2\cdots u_{2n})^{2k+1}}.
\end{align}

\begin{proposition}\label{prop_gcd}
	Let $r$ be a positive integer. Then, $$I_{2k+1}(ra_1,\dots, ra_{2n})=I_{2k+1}(a_1,\dots, a_{2n}).$$
\end{proposition}
\begin{proof}
	Denote the sum on the far right-hand side of  \eqref{higher_eq} by $L(a_1,\dots, a_{2n})$. We have
	\begin{align*} L(ra_1,\dots, ra_{2n})=&\sum_{\substack{u_i\in \mathbb{Z}\setminus\{0\}\\ \sum_{i=0}^{2n}ra_iu_i=0}}\frac{1}{(u_1u_2\cdots u_{2n})^{2k+1}}\\=&\sum_{\substack{u_i\in \mathbb{Z}\setminus\{0\}\\ \sum_{i=0}^{2n}a_iu_i=0}}\frac{1}{(u_1u_2\cdots u_{2n})^{2k+1}}=L(a_1,\cdots, a_{2n}).
\end{align*}
\end{proof}

\begin{proposition}\label{prop1}
	For any $a_i\in \mathbb{Q}$, any $k\in \mathbb{N}$, $n\in \mathbb{N}_{\geq 1}$, the sum
	\begin{align}\label{sum}
		\sum_{\substack{u_i\in \mathbb{Z}\setminus\{0\}\\ \sum_{i=0}^{2n}a_iu_i=0}}\frac{1}{(u_1u_2\cdots u_{2n})^{2k+1}}
	\end{align}
is a rational multiple of $\pi^{2n(2k+1)}$.
\end{proposition}

\begin{proof}
	Since the integrand of  $I_{2k+1} (a_1,a_2,\dots, a_{2n})$ is a product of polynomials, it follows that $I_{2k+1} (a_1,a_2,\cdots, a_{2n})$  is a rational number. Therefore, Proposition \ref{prop1} is a consequence of  \eqref{higher_eq}.
\end{proof}

\begin{cor}\label{reciprocal_sum}
	Let $A=(a_{ij}) \in \mathrm{GL}_{2n}(\mathbb{Z})$, and let $L_i$ denote the linear forms defined by the rows of the matrix $A$, i.e. $L_i= a_{i,1} u_1 +a_{i,2} u_2+\cdots +a_{i,n }  u_n$. Then for any $c_i\in \mathbb{Q}$, any  $k\in \mathbb{N}$, and $n\in \mathbb{N}_{\geq 1}$, the sum
	\begin{align}\label{sum2}
		\sum_{\substack{u_i\in \mathbb{Z}\setminus\{0\}\\ \sum_{i=0}^{2n}c_iu_i=0}}\frac{1}{(L_1L_2\cdots L_{2n})^{2k+1}}
	\end{align} is a rational multiple of $\pi^{2n(2k+1)}$.
\end{cor}
\begin{proof}
	Let $B=(b_{ij})=A^{-1}$ and use the fact $
	B\begin{pmatrix}
		L_1\\
		L_2\\
		\vdots\\
		L_{2n}
	\end{pmatrix}= \begin{pmatrix}
	u_1\\
	u_2\\
	\vdots\\
	u_{2n}
\end{pmatrix}$ to write the sum \eqref{sum2} in the form of \eqref{sum}. Let $b_i=c_i\sum _{k=1}^{2n}  b_{ji}$. Then the sum \eqref{sum2} is equal to $I(b_1,b_2,\dots , b_{2n})$, which is  a rational multiple of $\pi^{2n(2k+1)}$ by Proposition \ref{prop1}.
\end{proof}
Theorem \ref{theorem_higher_bernoulli} below gives an evaluation of the denominator of $I_{2k+1} (a_1,a_2,\dots, a_{2n})$ for $k\geq 1$, which  takes a different form from Theorem \ref{mcintoshgen} due to the presence of a constant term in the first Bernoill polynomial $B_1(x)=x-\frac{1}{2}$. We refer to Theorem 9, \cite{power-sum} for the value of the denominator of the polynomial $B_n(x)-B_n$ where $B_n$ is the $n$-th Bernoulli polynomial.

\begin{theorem}\label{theorem_higher_bernoulli}
Let $k\in \mathbb{N}_{\geq 1}$, $n\in \mathbb{N}_{\geq 1}$,  $\beta$ be  the least common multiple of the denominators  of the Bernoulli polynomials $B_{2k+1}(x)$ and $B_{(2k+1)(2n-1)+\alpha+1}(x)-B_{(2k+1)(2n-1)+\alpha+1},1\leq \alpha\leq 2k+1$, and
\begin{equation}\label{B}
 B:=lcm\left(\frac{\left[(2k+1)2n+1\right]!}{\left[(2k+1)(2n-1)+1\right]!}\beta^{2n}, (2n+1)\beta^{2n},(2n+2)\beta^{2n},\dots , [(2k+1)(2n-1)+1]\beta^{2n}\right).
 \end{equation}
   Then,
	$$
	f:= B \left[\frac{ X_1^{2n-1}X_3^{2n-3}\cdots  X_{2n-1}}{X_2^2X_4^4\cdots  X_{2n}^{2n}}\right]^{2k+1}I_{2k+1} (a_1,a_2,\dots, a_{2n})$$
	is an integer, where $$X_m= \prod_{1\leq i_1<i_2<\dots < i_m\leq 2n}(a_{i_1},a_{i_2},\dots ,  a_{i_m})$$ is the product of all greatest common divisors of $a_{i_1},a_{i_2},\dots ,  a_{i_m}$ with $1\leq i_1<i_2<\dots < i_m\leq 2n$.
\end{theorem}

\begin{proof}
	The ($2k+1$)-th Bernoulli polynomial, defined by \eqref{bernoullipoly} above, can be represented  by \cite[p.~588, no.~24.2.5]{nist}
\begin{gather} B_{2k+1}(x)=\label{ber}\\   x^{2k+1} + \binom{2k+1}{2k} B_1 x^{2k} +\binom{2k+1}{2k-1} B_2 x^{2k-1} +\cdots +\binom{2k+1}{3} B_{2k-2} x^3+\binom{2k+1}{1} B_{2k} x.\notag
\end{gather}
The integrand of \eqref{bernoulliintegral} has discontinuities at $\frac{k}{a_i},1\leq i\leq a_{i} -1$. Let $S=\{ \frac{k}{a_i},1\leq i\leq a_{i} -1\}$ and order the elements in $S$ in increasing order so that $S=\{x_1,x_2,\dots, x_N\}$. On the subinterval $(x_{l},x_{l+1})$, the integrand has the form
\begin{align}\label{integrand}
	P_l(x):=\prod_{i=0}^ {2n}B_{2k+1} (a_ix-n_{l_i}),
\end{align}
where, by \eqref{ber},
 $B_{2k+1} (a_ix-n_{l_i})$ is given by
\begin{align}	\label{bernoulli}
	B_{2k+1} (a_ix-n_{l_i} )=  \sum _{j=0}^{2k+1}  \binom{2k+1}{j}B_{2k+1-j} (a_ix-n_{l_i})^j.
\end{align}

Write
\begin{align}\label{bernoulli_expanded}
	B_{2k+1} (a_ix-n_{l_i} )=(a_ix) ^{2k+1 }  +D_{2k }  (a_ix) ^{2k} +D_{2k-1} (a_ix ) ^{2k-1 }+\cdots +D_1(a_ix),
 \end{align}
where $D_i$, $1\leq i \leq 2k$, are  rational numbers obtained by expanding \eqref{bernoulli}, whose denominators divide $\beta$.
Using \eqref{bernoulli_expanded}, we can rewrite \eqref{integrand} as
\begin{align}
	P_l(x)=&(a_1a_2\cdots a_{2n})^{2k+1} x^{(2k+1)2n}+ D_{2k} \left[\sum_{j=1}^{2n} a_1^{2k+1}a_2^{2k+1} \cdots a_j^{2k} \cdots a_{2n} ^{2k+1}\right]x^{(2k+1) 2n-1} + \nonumber  \\& \cdots +D_1^{2n} a_1\cdots a_{2n} x^{2n} .
\end{align}
Its anti-derivative is given by
\begin{align}
	F_l(x) :=&\frac{(a_1a_2\cdots a_{2n} )  ^{2k+1 }}{(2k+1) 2n+1  }  x^{(2k+1) 2n+1}+ \frac{D_{2k}}{(2k+1) 2n  }   \left[\sum _{j=1} ^{2n} a_1^{2k+1} \cdots a_j^{2k}\cdots a_{2n} ^{2k+1} \right]x^{(2k+1) 2n}\nonumber \\&+ \cdots +\frac{1}{2n+1} D_1^{2n} a_1\cdots a_{2n} x^{2n+1}.
\end{align}

Let $$L:=  \mathrm{lcm}(a_1,a_2\dots, a_{2n}),$$ and
 $$M:= B \left[\frac{ X_1^{2n-1}X_3^{2n-3}\cdots  X_{2n-1}}{X_2^2X_4^4\cdots  X_{2n}^{2n}}\right]^{2k+1},$$
 where $B$ is defined by \eqref{B}.
We want to show that $LM \,I_{2k+1} (a_1,a_2,\dots, a_{2n}) $ is an integer, and that it is divisible by $a_i, 1\leq i \leq 2n, $ separately. By Proposition \ref{prop_gcd}, we may assume  that $(a_1,a_2,\dots, a_{2n})=1$.

\begin{lemma}\label{lemma3.1}
	$LM\, I_{2k+1} (a_1,a_2,\dots, a_{2n}) $ is an integer.
\end{lemma}

\begin{proof}
Since $I_{2k+1} (a_1,a_2,\dots, a_{2n}) $ is a sum of polynomials $F_l(x)$ evaluated at the endpoints with denominators $a_1,a_2,\dots, a_{2n}$, it suffices to show that $LM \,F_l(\frac{1}{a_1})$ is an integer. View the coefficient of $x^i$ in $F_j(x)$ as a polynomial in $a_1,\ \dots ,a_n$. The denominators of the coefficients of this polynomial divide
$$\mathrm{lcm} ((2n+1)\beta^{2n}, (2n+2)\beta^{2n},\dots,[(2k+1)(2n-1)+1]\beta^{2n}),$$
because the denominator of  each $D_i$ divides  $\beta$. Hence, it remains to consider monomials in $a_1,a_2,\dots, a_{2n}$ appearing in the coefficients of $F_l(x)$. The monomial term in the coefficient of $x^m$ is of the form $a_1^{i_1} a_2^{i_2} \cdots a_{2n}^{i_{2n}}$  with $\sum_{j=1}^{2n} i_j=m-1$ and $1\leq i_j\leq 2k+1$. Then $a_1^{i_1} a_2^{i_2} \cdots a_{2n}^{i_{2n}}x^m$ evaluated at $\frac{1}{a_1}$  has denominator
 $${a_1} {\left[\frac{a_1}{(a_1,a_2)}\right]^{i_2}\cdots \left[\frac{a_1}{(a_1,a_{2n})}\right]^{i_{2n}}}.$$
 Since $i_j\leq 2k+1$ and
 $$\frac{X_1^{2n-1}}{X_2^2}=\prod_{i\neq j}\frac{x_i}{(x_i,x_j)}\frac{x_j}{(x_i,x_j)},$$ we see that   $${\left[\frac{a_1}{(a_1,a_2)}\right]^{i_2}\cdots \left[\frac{a_1}{(a_1,a_{2n})}\right]^{i_{2n}}}$$
 divides the factor
 $$ \left[\frac{X_1^{2n-1}}{X_2^2}\right]^{2k+1} $$ in $M$. Moreover, $a_1$ divides $L$, and thus
 $$LM\, a_1^{i_1} a_2^{i_2} \cdots a_{2n}^{i_{2n}} \left(\frac{1}{a_1}\right)^m $$
  is an integer. Hence, we conclude that $LM \,F_l(\frac{1}{a_1}) $ is an integer.
\end{proof}

Next, to show that $LM \,I_{2k+1} (a_1,a_2,\dots, a_{2n}) $ is divisible by $a_i, 1\leq i\leq 2n$, it suffices to show that it is divisible by $a_1$. As in the proof of Lemma \ref{lemma3.1}, it suffices to show that $LM \,F_l(x)$ evaluated at $\frac{1}{a_i}$, for $1\leq i\leq 2n$, is divisible by $a_1$.

\begin{lemma}\label{lemma3.2}
	$LM\, F_l(\frac{1}{a_i})$ is divisible by $a_1$ for $i\neq 1$.
\end{lemma}

\begin{proof}
    By symmetry, it suffices to show that $LM F_l(\frac{1}{a_2})$ is divisible by $a_1$.  As in the proof of Lemma \ref{lemma3.1}, the denominators of  the coefficients of the monomials in $a_1,a_2,\dots, a_{2n}$ are cancelled out by $B$. It suffices to consider the  monomials in the coefficients of $x^m$  that are of the form $a_1^{i_1} a_2^{i_2} \cdots a_{2n}^{i_{2n}}$  with $\sum_{j=1}^{2n} i_j=m-1$ and $1\leq i_j\leq 2k+1$. Then $a_1^{i_1} a_2^{i_2} \cdots a_{2n}^{i_{2n}}x^m$ evaluated at $\frac{1}{a_2}$  has denominator
    $${a_2} {\left[\frac{a_2}{(a_1,a_2)}\right]^{i_1}\cdots \left[\frac{a_2}{(a_1,a_{2n})}\right]^{i_{2n}}}.$$
     Multiplying by $LM$,
     we have $a_2\mid L$,  and
     $$\frac{Ma_1^{i_1} a_2^{i_2} \cdots a_{2n}^{i_{2n}}}{a_2^{m-1}}$$
      is an integer multiple of  \eqref{factor} which is divisible by $a_1$ by the proof of Lemma \ref{lemma2}. Therefore, $a_1$ divides
      $$\frac{LMa_1^{i_1} a_2^{i_2} \cdots a_{2n}^{i_{2n}}}{a_2^{m}}.$$
\end{proof}

\begin{lemma}\label{keylemma}
	$LM \,F_l(\frac{1}{a_1})$ is divisible by $a_1$.
\end{lemma}

\begin{proof}
	As before, it suffices to consider  the  monomials in the coefficients of $x^m$  which are of the form $a_1^{i_1} a_2^{i_2} \cdots a_{2n}^{i_{2n}}$,  with $\sum_{j=1}^{2n} i_j=m-1$ and $1\leq i_j\leq 2k+1$. Then $a_1^{i_1} a_2^{i_2} \cdots a_{2n}^{i_{2n}}x^m$ evaluated at $\frac{1}{a_1}$  has denominator $${a_1} {\left[\frac{a_1}{(a_1,a_2)}\right]^{i_2}\cdots \left[\frac{a_1}{(a_1,a_{2n})}\right]^{i_{2n}}}.$$
	
	If one of $i_2,\dots, i_{2n} <2k+1$, say $i_2$, then multiplying by $LM$, we see that
$$\frac{LMa_1^{i_1} a_2^{i_2} \cdots a_{2n}^{i_{2n}}}{a_1^m}$$
  is an integer multiple of  \eqref{factor3}, i.e.,
   $$\frac{a_1}{(a_1,a_2)}\frac{a_2}{(a_2,a_3)}  \frac{(a_1,a_2,a_3)}{(a_2,a_3,a_4,a_5)}\cdots  \frac{(a_1,\dots, a_{2k-1})}{(a_2,\cdots a_{2k})}.$$
  From the proof of Lemma \ref{lemma2}, this factor is divisible by $a_1$. Thus, it suffices to consider the term $i_2=i_3=\cdots =i_{2n}=2k+1$. Since subtracting terms divisible by $a_1$ does not change the divisibility by $a_1$,  one can ignore  the coefficients $D_{2k},\dots, D_1 $ in  $B_{2k+1}(a_ix-n_{l_i}), 2\leq i\leq 2n$, and it suffices to show that
  $$LM\int_0^1 \tilde{B}_{2k+1}(a_1 x) (a_2x)^{2k+1} \cdots (a_{2n}x)^{2k+1} dx$$
  is divisible by $a_1$.
	
	Since the denominators of the coefficients of $B_{2k+1}$ divide $\beta$, and $\beta^{2n}\mid B$, it suffices to show that
	\begin{align}\label{integral_2k+1}
		LM\int_0^1 ((a_1 x))^\alpha (a_2x)^{2k+1} \cdots (a_{2n}x)^{2k+1} dx
	\end{align}
is divisible by $a_1$, for $1\leq \alpha\leq 2k+1$.
\newcommand{\pp}{(2k+1)(2n-1)}

Letting  $P:=(2k+1)(2n-1)$ and using an integration by parts,  we write \eqref{integral_2k+1} as
\begin{align}
	&LM(a_2\cdots a_{2n} ) ^{2k+1}\sum_{j=0}^{a_1-1} \int_{\frac{j}{a_1}}^\frac{j+1}{a_1}(a_1x-j)^\alpha d\left(\frac{x^{{P}+1}}{P+1}\right)\nonumber \\=&{LM(a_2\cdots a_{2n} ) ^{2k+1}}\left[\sum_{j=0}^{a_1-1}\frac{(j+1)^{P+1}}{(P+1)a_1^{P+1}}-\frac{a_1\alpha}{P+1} \sum_{j=0}^{a_1-1}\int_{\frac{j}{a_1}}^\frac{j+1}{a_1} (a_1x-j)^{\alpha-1}x^{P+1}\,dx\right]\nonumber\\
	=&{LM(a_2\cdots a_{2n} ) ^{2k+1}}\left[\frac{B_{P+2}(a_1) -B_{P+2}}{(P+1)(P+2)a_1^{P+1}}-\frac{a_1\alpha}{P+1} \sum_{j=0}^{a_1-1}\int_{\frac{j}{a_1}}^\frac{j+1}{a_1} (a_1x-j)^{\alpha-1}x^{P+1}\,dx\right].
\end{align}

Continuing with integrations by parts,  we note that at each step the integrated terms have denominators dividing $\frac{\left[(2k+1)2n+1\right]!}{\left[(2k+1)(2n-1)+1\right]!}\beta$, which divides $B$.  Thus, it suffices to show that

\begin{align}\label{integral_final}
	\mathbb{L}:=LM(a_2\cdots a_{2n} ) ^{2k+1}\frac{a_1^{\alpha-1 }\alpha!}{(P+1)(P+2)\cdots(P+\alpha-1)}\sum_{j=0}^{a_1-1}\int_{\frac{j}{a_1}}^\frac{j+1}{a_1} (a_1x-j)x^{P+\alpha -1}\,dx
\end{align}
is divisible by $a_1$ for each $\alpha$, $1\leq \alpha\leq 2k+1$.

Evaluating \eqref{integral_final}, we deduce that
\begin{align*}
	\mathbb{L}=& \frac{LM(a_2\cdots a_{2n} )^{2k+1}a_1 ^{\alpha-1}\alpha!}{(P+1) (P+2)\cdots (P+\alpha -1)}\\
&\times\left\{\frac{ a_1}{P+\alpha+1}-\sum_{j=1}^{a_1-1} \frac{j}{P+\alpha}\left[\left(\frac{j+1}{a_1}\right)^{P+\alpha}-\left(\frac{j}{a_1}\right)^{P+\alpha}\right]\right\}.
\end{align*}
The first expression on the right-hand side above is divisible by $a_1$, since $(P+1)(P+2)\cdots (P+\alpha-1)(P+\alpha+1)$ divides $M$ and $a_1$ divides $L$.
The second expression is equal to
\begin{align}\label{integralfinal_2}
&	- \frac{LM(a_2\cdots a_{2n} )^{2k+1}a_1 ^{\alpha-1}\alpha!}{(P+1) (P+2)\cdots (P+\alpha -1)(P+\alpha )a_1^{P+\alpha}}\sum
_{j=0}^{a_1-1} j\left[(j+1) ^{p+\alpha} -j^{p+\alpha}\right]\nonumber \\
=&- \frac{LM(a_2\cdots a_{2n} )^{2k+1}a_1 ^{\alpha-1}\alpha!}{(P+1) (P+2)\cdots (P+\alpha -1)(P+\alpha )a_1^{P+\alpha}}\left[a_1^{P+\alpha}(a_1-1)-\sum_{j=1} ^{a_1-1}j^{p+\alpha}\right].
\end{align}
The first term on the right-hand side of \eqref{integralfinal_2} is equal to
 $$- \frac{LM(a_2\cdots a_{2n} )^{2k+1}a_1 ^{\alpha-1}\alpha!(a_1-1)}{(P+1) (P+2)\cdots (P+\alpha )},$$
   which is divisible by $a_1$, since ${(P+1) (P+2)\cdots (P+\alpha )}$ divides $M$ and $a_1$ divides $L$. With the aid of \eqref{sumofpowers}, the second term on the right-hand side of \eqref{integralfinal_2} is equal to
\begin{align}\label{second_term}
	\frac{LM(a_2\cdots a_{2n} )^{2k+1}a_1 ^{\alpha-1}\alpha!}{(P+1) (P+2)\cdots (P+\alpha )a_1^{P+\alpha}}\frac{B_{P+\alpha+1}(a_1)-B_{P+\alpha+1}}{(P+\alpha+1)}.
\end{align}
Again, we observe that $B_{P+\alpha+1}(a_1)-B_{P+\alpha+1}$ is a rational polynomial in $a_1$ without a constant term. Recalling that $P=(2k+1)(2n-1)$, we see  that  \eqref{second_term} is an integer multiple of
 $$\frac{LM}{\beta(P+1) (P+2)\cdots (P+\alpha+1)  \left[\prod_{j=2}^{2n}\frac{a_1}{(a_1,a_j)}\right]^{2k+1}}.$$
  Now we see that $\beta(P+1) (P+2)\cdots (P+\alpha+1) $ divides the constant $B$ in $M$, and that $\left[\prod_{j=2}^{2n}\frac{a_1}{(a_1,a_j)}\right]^{2k+1}$ divides the factor $\left[\frac{X_1^{2n-1}}{X_2^2}\right]^{2k+1}$ in $M$. Since $a_1\mid L$, therefore $a_1$ divides  \eqref{second_term}.

  Therefore, we have proved Lemma \ref{keylemma} and thereby also  Theorem \ref{theorem_higher_bernoulli}.
	
\end{proof}
\end{proof}

\end{document}